\documentclass[onefignum,onetabnum]{siamart220329}
\pdfoutput=1

\usepackage{algorithm}
\usepackage{algpseudocode} 
\definecolor{lo}{HTML}{33A02C}
\definecolor{med}{HTML}{1F78B4}
\definecolor{hi}{HTML}{E31A1C}

\headers{Rounding-Error Analysis of Multigrid $V$-Cycles}{S. F. McCormick and R. Tamstorf}

\title{Rounding-Error Analysis of Multigrid $V$-Cycles}

\author{Stephen F. McCormick\thanks{University of Colorado at Boulder, Boulder, CO
  (\email{stephen.mccormick@colorado.edu}).}
\and Rasmus Tamstorf\thanks{Walt Disney Animation Studios, Burbank, CA
  (\email{rt@acm.org}).}
}

\usepackage{amsmath,amssymb,amsopn}
\usepackage[T1]{fontenc}

\definecolor{darkblue}{rgb}{0.0,0.0,0.3}

\usepackage{doi}

\newtheorem{thm}{Theorem}
\newtheorem{rem}{Remark}
\DeclareMathOperator{\fl}{fl}
\newcommand{\ewed}{\boldsymbol{\dot{\varepsilon}}}
\newcommand{\ewe}{\boldsymbol{\varepsilon}}
\newcommand{\eweb}{\boldsymbol{\bar{\varepsilon}}}
\renewcommand{\Re}{\mathbb{R}}
\newcommand{\aladot}{{\dot{m}}_A^+}
\newcommand{\alpdot}{{\dot{m}}_P^+}
\newcommand{\alndot}{{\dot{m}}_N^+}
\newcommand{\pvar}{\tau}
\newcommand{\pvard}{{\dot{\pvar}}}
\newcommand{\pvarb}{{\bar{\pvar}}}
\newcommand{\tg}{$\mathcal{TG}$}
\newcommand{\ir}{$\mathcal{IR}$}
\newcommand{\fmg}{$\mathcal{FMG}$}

\begin{document}

\maketitle

\begin{abstract}
  This paper provides a rounding-error analysis for two-grid methods that use one relaxation step both before and after coarsening. The analysis is based on floating point arithmetic and focuses on a two-grid scheme that is perturbed on the coarse grid to allow for an approximate coarse-grid solve. Leveraging previously published results, this two-grid theory can then be extended to general $V(\mu,\nu)$-cycles, as well as full multigrid ({\fmg}). It can also be extended to mixed-precision iterative refinement ({\ir}) based on these cycles. An added benefit of the theory here over previous work is that it is obtained in a more organized, transparent, and simpler way.
\end{abstract}



\section{Introduction and Notation}

Our aim is to analyze the effects of applying floating-point arithmetic to multigrid methods for computing the solution $y \in \Re^n$ of
\begin{equation}
Ay = r, 
\label{ayr}
\end{equation}
where $r \in \Re^n$ is given and $A \in \Re^{n \times n}$ is a given symmetric positive definite matrix, with at most $m_A$ nonzeros per row. In concert with the iterative refinement application of multigrid described in \cite{McCormick2021}, our focus is on a two-grid ({\tg }) cycle applied to (\ref{ayr}) that starts with a zero initial guess and uses both one {\em pre-relaxation} (i.e., before coarsening) and one {\em post-relaxation} (i.e., after coarsening). This allows all of the theory in \cite{McCormick2021} and \cite{Tamstorf2021} to be extended to a general $V(\mu, \nu)$-cycle. We assume that the reader is familiar with this earlier theory, but emphasize that the analysis here represents an improvement in terms of organization, transparency, and simplicity.

Assume that $A$ has been computed exactly in  ``low'' $\ewed$-precision as in \cite{McCormick2021} and, for simplicity, that $A$ is scaled so that $\|A\| = 1$, where $\|\cdot\|$ denotes the Euclidean norm. Assume also, as in \cite{McCormick2021}, that the interpolation matrix $P \in \Re^{n \times n_c}, n_c < n,$ and the Galerkin coarse-level matrix $A_c = P^tAP$ are exact to $\ewed$-precision, and that $P$ has at most $m_P$ nonzeros in any row or column and is scaled so that $\|A_c\| = 1$. The two-grid cycle is assumed to use pre-relaxation and post-relaxation given by the stationary linear iterations $x \leftarrow x - M (Ax - b)$ and $x \leftarrow x - N (Ax - b)$, respectively. Here, $M, N \in \Re^{n \times n}$ are easily computed nonsingular matrices that approximate $A^{-1}$ well enough to guarantee convergence in energy in that $\|I - M A\|_A < 1$ and $\|I - N A\|_A < 1$,  where $I$ is the identity. ($I_c$ is used below to denote the coarse-level identity.) Coarsening in {\tg } consists of an exact  solve of the Galerkin coarse-level problem that has been perturbed to allow for a recursive solve as in \cite{McCormick2021}. We specify this perturbed coarse-grid correction in {\tg } as $B_c A_c^{-1} P^t r$, where $B_c \in \Re^{n_c \times n_c}$ satisfies $\|B_c - I_c\|_{A_c} < 1$. 

We complete this section with additional notation. 
Denote the condition numbers of $A$ and $A_c$ by $\kappa = \|A\| \cdot \|A^{-1}\| =  \|A^{-1}\|$ and $\kappa_c = \|A_c\| \cdot \|A_c^{-1}\| = \|A_c^{-1}\|$, respectively. Define the energy norm by $\|\cdot\|_A = \|A^\frac{1}{2} \cdot\|$ and let
\[
\eta_A= \||A|\|, \,\, \eta_P = \||P|\|, \,\, \eta_M= \|M\|, \,\, \eta_N = \|N\|_A, 
\]
\[
\, \aladot = \frac{m_A + 1}{1 - (m_A + 1){\ewed}} , \,\, \alpdot = \frac{m_P + 1}{1 - (m_P + 1)\ewed} , \textrm{   and   }\alndot = \frac{m_N + 1}{1 - (m_N + 1)\ewed} .
\]
The measures used here for $M$ and $N$ are the respective Euclidean and energy norms. The Euclidean norm for pre-relaxation is a departure from what is used in \cite{McCormick2021} because it allows us to avoid the more cumbersome energy norm in the earlier stages of the proof; see (\ref{dr})-({\ref{dpm}). For Krylov methods, these two measures are the same because $\|M\|_A = \|A^\frac{1}{2}MA^{-\frac{1}{2}}\| = \|M\|$ when $M$ is a polynomial in $A$.

To account for rounding errors in relaxation, suppose that constants $\alpha_M$ and $\alpha_N$ exist such that computing $M z$ and $N z$ for a vector $z \in \Re^n$ in $\ewed$-precision yields the respective results
\begin{equation}
M  z + \delta_M, \, \|\delta_M\| \le \alpha_M \ewed \|z\|, \quad \textrm{and} \quad
N  z + \delta_N, \, \|\delta_N\| \le \alpha_N \ewed \|z\|.
\label{mest}
\end{equation}

Because our error bounds are posed in energy while our basic precision estimates are in the Euclidean norm, the analysis that follows makes frequent use of the following inequalities that are assumed to hold for any $w\in\Re^n$:
\begin{align}
\|w\| &= \|AA^{-1}w\| \le \|A^\frac{1}{2}\| \cdot \|A^\frac{1}{2}A^{-1}w\|= \|A^{-1}w\|_A , 
\label{r}\\
\|w_c\| &= \|A_c^{-\frac{1}{2}}A_c^\frac{1}{2}w_c\| \le \|A_c^{-\frac{1}{2}}\| \cdot \|A_c^\frac{1}{2}w_c\|= \kappa_c^\frac{1}{2} \|w_c\|_{A_c} , 
\label{rc}
\end{align}
and
\begin{equation}
\|w\|_A = \|A^\frac{1}{2}w\| \le \|w\| = \|A^{-\frac{1}{2}} A^\frac{1}{2} w\| \le \|A^{-\frac{1}{2}}\| \cdot \|w\|_A = \kappa^\frac{1}{2} \|w\|_A 
\label{e2E}
\end{equation}

The remainder of the paper states basic rounding-error estimates in Section~\ref{models}, followed by a description of {\tg } in Section~\ref{tg}. The two-grid theory is then presented in Section~\ref{analysis}, and concluding remarks about extending {\tg } to general $V(\mu, \nu)$-cycles, full multigrid (FMG), and progressive precision are given in Section~\ref{conclusion}.

\section{Rounding-Error Models}
\label{models}
Floating-point error estimates in this section are taken from \cite[Section 2]{Higham2002}.

Quantization error in $\ewed$-precision for $w \in \Re^n$ obeys the estimate
\begin{equation}
\fl(w) = w + \delta, \quad \|\delta\| \le \ewed \|w\| ,
\label{quant}
\end{equation}
where $\fl(\cdot)$ denotes the computed result. With $v, w \in \Re^n$, we also have that
\begin{equation}
\fl(v \pm w) = v \pm w + \delta, \quad \|\delta\| \le \ewed \|v \pm w\| .
\label{pm}
\end{equation}
Let the absolute value of a vector or matrix denote the respective vector or matrix of absolute values. For any $K \in \Re^{n \times n}$ and with $m_K$ and $\dot{m}^+_K$ defined in analogy to the above, the model for $\ewed$-precision computation of the residual $Kw - c$ is then given by
\begin{equation}
\fl(Kw - c) = Kw - c + \delta, \quad \|\delta\| \le \ewed \dot{m}^+_K (\|c\| + \||K|\| \cdot \|w\|).
\label{res}
\end{equation}
With $P$ and $P^t$ in mind, note that (\ref{res}) reduces loosely for $K$ in $\Re^{m \times n}$ or $\Re^{n \times m}$ to
\begin{equation}
\fl(Kw) = Kw + \delta, \quad \|\delta\| \le \ewed \dot{m}^+_K \||K|\| \cdot \|w\|.
\label{ax}
\end{equation}

Instead of using $\fl$ below, we add $\delta$'s to exact expressions to denote quantities computed in finite precision. For example, the {\em computed} residual in (\ref{res}) might be written as $Kw - c + \delta_r$ while the {\em computed} solution of $Kw = c$ might be $K^{-1}c + \delta$.

\section{Two-Grid Cycle}
\label{tg}

The pseudocode for the two-grid cycle, {\tg}, is shown in Algorithm \ref{alg-tg} below. 
All computations in this algorithm are performed in ``low'' $\ewed$-precision (green font), except for the exact, infinite-precision coarse-level solve (black font). Accordingly, since the input right-hand side (RHS) may be in higher precision, the cycle is initialized with a rounding step. The solver then proceeds by relaxing on the zero initial guess to the solution of (\ref{ayr}), improving the result by a coarse-level correction based on prolongation, and then applying one more relaxation step. We use subscripts $\mu$ and $\nu$ here and in the next section to identify quantities occurring respectively before and after coarse-grid correction.

\noindent
\begin{center}
\begin{minipage}{.8\linewidth}
\begin{algorithm}[H]
\footnotesize
\caption{$\texttt{Two-Grid (\tg) Correction Scheme}$}
\label{alg-tg}
\begin{algorithmic}[1]
\Require A, r, P, M, N.
\State $\textcolor{lo}{r \leftarrow} r$ \Comment{\textcolor{lo}{Round RHS} and Initialize {\tg }}\label{tg:round}
\State \textcolor{lo}{$y_\mu \leftarrow  M r$}\Comment{\textcolor{lo}{Pre-relax on Current Approximation $y=0$}}\label{tg:prerelax}
\State \textcolor{lo}{$r_\mu \leftarrow Ay_\mu - r$}\Comment{\textcolor{lo}{Evaluate {\tg } Residual for Pre-relaxed $y_\mu$}}\label{tg:preresidual}
\State \textcolor{lo}{$r_c \leftarrow P^tr_\mu$}\Comment{\textcolor{lo}{Restrict {\tg } Residual to Coarse-Level}}\label{tg:restrict}
\State $d_c \leftarrow B_cA_c^{-1}r_c$\Comment{Compute Perturbed Coarse-Level Solution}\label{tg:coarsesolve}
\State \textcolor{lo}{$d \leftarrow Pd_c$}\Comment{\textcolor{lo}{Interpolate Correction to Fine Level}}\label{tg:prolong}
\State \textcolor{lo}{$y_\nu \leftarrow y_\mu - d$}\Comment{\textcolor{lo}{Apply Correction to Approximation $y_\mu$}}\label{tg:correction}
\State \textcolor{lo}{$r_\nu \leftarrow Ay_\nu - r$}\Comment{\textcolor{lo}{Evaluate {\tg } Residual for Corrected $y_\nu$}}\label{tg:postresidual}
\State \textcolor{lo}{$r_N \leftarrow N r_\nu$}\Comment{\textcolor{lo}{Precondition Residual}}\label{tg:precond}
\State \textcolor{lo}{$y \leftarrow y_\nu - r_N$}\Comment{\textcolor{lo}{Post-relax on Current Approximation $y_\nu$}}\label{tg:postrelax}
\State\Return $y$\Comment{Return Approximate Solution of $Ay = r$}
\end{algorithmic}
\end{algorithm}
\end{minipage}
\end{center}
\vspace{1em}

Assume for the rest of this section that {\tg } is performed in infinite precision, meaning that $y$ and the other quantities  are computed in exact arithmetic. A basic assumption we make is that {\tg } converges in infinite precision. Noting that the respective initial and final errors are $0 - A^{-1}r = - A^{-1}r$ and $y - A^{-1}r$, we assume that there exists a $\rho_{tg}^* < 1$ such that
\[
\|y - A^{-1}r\|_A \le \rho_{tg}^* \|A^{-1}r\|_A .
\]
Note then that
\begin{equation}
\|y\|_A \le \|y - A^{-1}r\|_A + \|A^{-1}r\|_A \le (1 + \rho_{tg}^*) \|A^{-1}r\|_A \le 2 \|A^{-1}r\|_A .
\label{y}
\end{equation}
The theory uses the fact that the intermediate result $y_\nu$ does not increase the initial energy error. The fact that this is true is a direct result of the bound just after equation (5.2) in \cite{McCormick2021}. For completeness, we derive this bound tailored to our setting. Specifically, since $T \equiv I - P(P^tAP)^{-1}P^tA$ is an energy-orthogonal projection onto the energy-orthogonal complement of $P$ (c.f. \cite{tutorial}), it follows that the error propagation matrix connecting the initial error $- A^{-1}r$ to the intermediate error $y_\nu - A^{-1}r$ is bounded according to
\begin{align*}
\|(I - PB_c (P^tAP)^{-1}P^tA)(I - MA)\|_A^2 
& \le \|(I - PB_c (P^tAP)^{-1}P^tA)\|_A^2  \|I - MA\|_A^2 \\
& \le \|(I - PB_c (P^tAP)^{-1}P^tA)\|_A^2 \\
& = \|T\|_A^2 + \|P(B_c - I_c) (P^tAP)^{-1}P^tA\|_A^2  \\
& \le \|T\|_A^2 + \|B_c - I_c\|_{A_c}^2  \|(P^tAP)^{-1}P^tA\|_{A_c}^2  \\
& \le \|T\|_A^2 + \|I - T\|_A^2 = 1.
\end{align*}
We can thus conclude that
\begin{equation}
\|y_\nu - A^{-1}r\|_A \le \|A^{-1}r\|_A \textrm{   and   }\|y_\nu\|_A \le \|A^{-1}r\|_A + \|y_\nu - A^{-1}r\|_A \le 2 \|A^{-1}r\|_A .
\label{ynu}
\end{equation}
Note also that
\[
\|A^{-1}r_\mu\|_A = \|y_\mu - A^{-1}r\|_A = \|(I - MA) A^{-1}r\|_A \le \|A^{-1}r\|_A 
\]
and, since $I - A^\frac{1}{2}PA_c^{-1}P^tA^\frac{1}{2}$ is an orthogonal projection onto the orthogonal complement of the range of $A^\frac{1}{2}P$, then $A^\frac{1}{2}PA_c^{-1}P^tA^\frac{1}{2} \le I$ and, hence,
\[
\|A_c^{-1}P^tr_\mu\|_{A_c} = \langle PA_c^{-1}P^tr_\mu, r_\mu\rangle^\frac{1}{2} \le \langle A^{-1}r_\mu, r_\mu\rangle^\frac{1}{2} = \|A^{-1}r_\mu\|_A \le \|A^{-1}r\|_A .
\]
This bound and the fact that $\|B_c\|_{A_c} \le \|B_c - I_c\|_{A_c} + \|I_c\|_{A_c} < 2$ imply that
\begin{equation}
\|d_c\| \le \kappa_c^\frac{1}{2} 
\|d_c\|_{A_c} = \kappa_c^\frac{1}{2} \|B_cA_c^{-1}r_c\|_{A_c} \le 2\kappa_c^\frac{1}{2} \|A_c^{-1}P^tr_\mu\|_{A_c}
\le 2 \kappa_c^\frac{1}{2} \|A^{-1}r\|_A .
\label{dc}
\end{equation}

In the following, we make use of the above infinite-precision bounds to establish results for finite-precision computations.

\section{Finite Precision Two-Grid Theory}
\label{analysis}

The main theorem below makes use of the following parameters:
\begin{align*}
C_0 &=  (1 + \aladot (1 + \eta_A \eta_M) + \aladot \ewed + (1 + \aladot \ewed) \eta_A (\eta_M + \alpha_M  (1 + \ewed))) \ewed ,\\
C_1 &=  \kappa_c^\frac{1}{2} (\eta_P C_0 + \ewed \alpdot \eta_P (1 + \eta_A \eta_M + C_0)) ,\\
C_2 &= 2 \kappa_c^\frac{1}{2} \ewed \alpdot \eta_P (1 + C_1) + 2 C_1 , \\
C_3 &= C_2 + ((2 + C_2)\kappa^\frac{1}{2}  + (\eta_M + \alpha_M  (1 + \ewed))(1 + \ewed)) \ewed ,\\
C_4 &= \eta_A C_3 + \aladot ((\eta_A (2 + C_3) + 1) \kappa^\frac{1}{2} + \ewed) \ewed , \\
C_5 &= (2 + C_3 + \eta_N C_4 + \ewed (1 + \kappa^\frac{1}{2} C_4)) \kappa^\frac{1}{2} \ewed .
\end{align*} 
The complexity of these parameters is due to the many finite-precision steps that {\tg } requires. We simplify them somewhat in Remark~\ref{constants} by considering their asymptotic behavior with respect to $n$  for the applications we have in mind.

\begin{thm}{\bf Two-Grid Error Estimate.} Let $y$ denote the result of one{ \tg } cycle applied to (\ref{ayr}) using exact arithmetic. Then one{ \tg } cycle applied to (\ref{ayr}) using $\ewed-$precision arithmetic produces a result $y + \delta_y$ that satisfies
\begin{equation}
\|y+ \delta_y - A^{-1}r\|_A \le \rho_{tg} \|A^{-1}r\|_A, \quad \rho_{tg} = \rho_{tg}^* + \delta_{\rho_{tg}} ,
\label{conv}
\end{equation}
where $\delta_{\rho_{tg}} = C_3 + C_4 + C_5$.
\label{thm:tg}
\end{thm}

\begin{proof}
The proof is organized in an ordered list of the $\ewed$-precision errors as they occur, with the bounds accumulating the estimates along the way. We emphasize that the quantities used in Algorithm~\ref{alg-tg} are understood here to be in infinite precision; errors in the computed quantities are expressed instead by adding deltas. For each step in the proof, we refer to the corresponding line number in Algorithm~\ref{alg-tg}.
\begin{itemize}

\item Line~\ref{tg:round}. $\delta_r =$ error in $\ewed$ quantization of $r$.

From (\ref{quant}) and (\ref{r}):
\begin{equation}
\|\delta_r\| \le \ewed \|r\| \le \ewed \|A^{-1}r\|_A .
\label{dr}
\end{equation}

\item Line~\ref{tg:prerelax}. $\delta_\mu =$ error in  $M(r + \delta_r)$ given $r + \delta_r$.

From (\ref{mest}), (\ref{r}), and (\ref{dr}):
\begin{equation}
\|\delta_\mu\| \le \alpha_M \ewed (\|r\| + \|\delta_r\|) \le \alpha_M  (1 + \ewed) \ewed \|A^{-1}r\|_A .
\label{dm}
\end{equation}

\item Line~\ref{tg:prerelax}. $\delta_{y_{\mu}} = M \delta_r + \delta_\mu$, accumulated error in  $y_\mu = M r$.

From (\ref{dr}) and (\ref{dm}):
\begin{equation}
\|\delta_{y_\mu}\| \le \eta_M \|\delta_r\| + \|\delta_\mu\| \le (\eta_M + \alpha_M  (1 + \ewed)) \ewed \|A^{-1}r\|_A .
\label{dym}
\end{equation}

\item Line~\ref{tg:preresidual}. $\delta_{a_\mu} =$ error in  $A (y_\mu + \delta_{y_\mu}) - (r + \delta_r)$ given $y_\mu + \delta_{y_\mu}$ and $r + \delta_r$.

From (\ref{res}), (\ref{r}), (\ref{dr}), and (\ref{dym}):
\begin{align}
\label{dam}
\|\delta_{a_\mu}\| &\le \ewed \aladot (\|r\| + \|\delta_r\| + \eta_A (\|y_\mu\| + \|\delta_{y_\mu}\|)) \\
&\le \ewed \aladot ((1 + \eta_A \eta_M) \|r\| + \|\delta_r\| + \eta_A \|\delta_{y_\mu}\|) \nonumber \\
&\le \aladot (1 + \eta_A \eta_M + \ewed + \eta_A (\eta_M + \alpha_M  (1 + \ewed)) \ewed) \ewed \|A^{-1}r\|_A \nonumber .
\end{align}

\item Line~\ref{tg:preresidual}. $\delta_{r_{\mu}} = A \delta_{y_\mu} - \delta_r + \delta_{a_\mu}$, accumulated error in  $r_\mu = A y_\mu - r$.

From (\ref{dym}), (\ref{dr}), and (\ref{dam}):
\begin{equation}
\|\delta_{r_{\mu}}\| \le \eta_A \|\delta_{y_\mu}\| + \|\delta_r\| + \|\delta_{a_\mu}\| 
\le C_0 \|A^{-1}r\|_A .
\label{drm}
\end{equation}

\item Line~\ref{tg:restrict}. $\delta_{p_\mu} =$ error in  $P^t(r_{\mu} + \delta_{r_{\mu}})$ given $r_{\mu} + \delta_{r_{\mu}}$.

From (\ref{ax}), (\ref{drm}), and (\ref{r}):
\begin{align}
\label{dpm}
\|\delta_{p_\mu}\| &\le \ewed \alpdot \eta_P (\|r_{\mu}\| + \|\delta_{r_{\mu}}\|) \\
&\le \ewed \alpdot \eta_P (\|AMr\| + \|r\| + C_0) \|A^{-1}r\|_A \nonumber \\
&\le \ewed \alpdot \eta_P (1 + \eta_A \eta_M + C_0) \|A^{-1}r\|_A \nonumber .
\end{align}

\item Line~\ref{tg:coarsesolve}. $\delta_{d_c} = B_c A_c^{-1} (P^t\delta_{r_\mu} + \delta_{p_\mu})$, propagated error in  $d_c \equiv B_c A_c^{-1} P^tr_c$. 

From (\ref{drm}) and (\ref{dpm}):
\begin{equation}
\|\delta_{d_c}\|_{A_c} \le 2 \|A_c^{-1}(P^t\delta_{r_\mu} + \delta_{p_\mu})\|_{A_c} \le 2 \kappa_c^\frac{1}{2} (\eta_P \|\delta_{r_\mu}\| + \|\delta_{p_\mu}\|)\| \le 2 C_1 \|A^{-1}r\|_A .
\label{C1}
\end{equation}

\item Line~\ref{tg:prolong}. $\delta_{p_\nu} =$ error in  $P(d_c + \delta_{d_c})$ given $d_c + \delta_{d_c}$.

From (\ref{e2E}) twice, (\ref{ax}), (\ref{dc}), and (\ref{C1}):
\begin{equation}
\|\delta_{p_\nu}\|_A \le \ewed \alpdot \eta_P (\|d_c\| + \|\delta_{d_c}\|)
\le 2 \kappa_c^\frac{1}{2} \ewed \alpdot \eta_P (1 + C_1) \|A^{-1}r\|_A .
\label{dpn}
\end{equation}

\item Line~\ref{tg:prolong}. $\delta_{d} = P\delta_{d_c} + \delta_{p_\nu}$, accumulated error in  $d \equiv Pd_c$.

From (\ref{C1}) and (\ref{dpn}):
\begin{equation}
\|\delta_{d}\|_A \le \|P\delta_{d_c}\|_A + \|\delta_{p_\nu}\|_A = \|\delta_{d_c}\|_{A_c} + \|\delta_{p_\nu}\|_A \le C_2 \|A^{-1}r\|_A .
\label{C2}
\end{equation}

\item Line~\ref{tg:correction}. $\delta_{y_-} =$ error in  $(y_\mu + \delta_{y_\mu}) - (d + \delta_{d})$ given $y_\mu + \delta_{y_\mu}$ and  $d + \delta_{d}$.

From (\ref{pm}), (\ref{e2E}) twice, (\ref{ynu}), (\ref{C2}), and (\ref{dym}):
\begin{align}
\label{dyminus}
\|\delta_{y_-}\| &\le \ewed \|(y_\mu + \delta_{y_\mu}) - (d + \delta_{d})\| \\
&\le \ewed (\|y_\mu - d\| + \|\delta_{y_\mu}\| + \|\delta_{d}\|) \nonumber \\
&\le \ewed \kappa^\frac{1}{2} (\|y_\nu\|_A + \|\delta_{d}\|_A) + \ewed \|\delta_{y_\mu}\| \nonumber \\
&\le ((2 + C_2)\kappa^\frac{1}{2}  + (\eta_M + \alpha_M  (1 + \ewed))\ewed) \ewed  \|A^{-1}r\|_A \nonumber . 
\end{align}

\item Line~\ref{tg:correction}. $\delta_{y_\nu} = \delta_{y_\mu} - \delta_{d} + \delta_{y_-}$, accumulated error in  $y_\nu \equiv y_\mu - d$.

From (\ref{dym}), (\ref{C2}), and (\ref{dyminus}):
\begin{equation}
\|\delta_{y_\nu}\|_A \le \|\delta_{y_\mu}\|_A + \|\delta_{d}\|_A + \|\delta_{y_-}\|_A \le C_3 \|A^{-1}r\|_A .
\label{C3}
\end{equation}

\item Line~\ref{tg:postresidual}. $\delta_{a_\nu} =$ error in  $A(y_\nu + \delta_{y_\nu}) - (r + \delta_r)$ given $y_\nu + \delta_{y_\nu}$ and $r + \delta_r$.

From (\ref{r}) four times, (\ref{res}), (\ref{dr}), (\ref{ynu}), (\ref{C3}), and (\ref{r}):
\begin{align}
\label{dan}
\|\delta_{a_\nu}\|_A  
&\le \ewed \aladot (\eta_A (\|y_\nu\| + \|\delta_{y_\nu}\|) + \|r\| + \|\delta_r\|) \\
&\le \ewed \aladot (\kappa^\frac{1}{2} (\eta_A (\|y_\nu\|_A + \|\delta_{y_\nu}\|_A) + \|r\|_A) + \ewed \|A^{-1}r\|_A) \nonumber \\
&\le \aladot ((\eta_A (2 + C_3) + 1) \kappa^\frac{1}{2} + \ewed) \ewed \|A^{-1}r\|_A \nonumber .
\end{align}

\item Line~\ref{tg:postresidual}. $\delta_{r_\nu} = A\delta_{y_\nu} + \delta_{a_\nu}$, accumulated error in  $r_\nu \equiv Ay_\nu  - r$.

From (\ref{C3}) and (\ref{dan}):
\begin{equation}
\|\delta_{r_\nu}\|_A \le \eta_A \|\delta_{y_\nu}\|_A + \|\delta_{a_\nu}\|_A \le C_4 \|A^{-1}r\|_A  .
\label{C4}
\end{equation}

\item Line~\ref{tg:precond}. $\delta_\nu =$ error in  $N(r_\nu + \delta_{r_\nu})$ given $r_\nu + \delta_{r_\nu}$.

From (\ref{mest}), (\ref{ynu}) twice, (\ref{C4}):
\begin{align}
\|\delta_\nu\| &\le \alpha_N \ewed (\|r_\nu\| + \|\delta_{r_\nu}\|) \label{N} \\
&\le \alpha_N \ewed (\|A(y_\nu - A^{-1}r)\| + \kappa^\frac{1}{2} \|\delta_{r_\nu}\|_A) \nonumber \\
&= \alpha_N \ewed (\|y_\nu - A^{-1}r\|_A + \kappa^\frac{1}{2} \|\delta_{r_\nu}\|_A) \nonumber \\
&\le \alpha_N \ewed (1 + \kappa^\frac{1}{2} C_4) \|A^{-1}r\|_A \nonumber.
\end{align}

\item Line~\ref{tg:precond} $\delta_{r_N} = N\delta_{r_\nu} + \delta_\nu$ accumulated error in  $r_N \equiv N r_\nu$.

From (\ref{C4}), (\ref{N}):
\begin{equation}
\|\delta_{r_N}\|_A \le \eta_N \|\delta_{r_\nu}\|_A + \|\delta_\nu\| \le (\eta_N C_4 + \ewed (1 + \kappa^\frac{1}{2} C_4)) \|A^{-1}r\|_A  .
\label{rN}
\end{equation}

\item Line~\ref{tg:postrelax}. $\delta_\nu =$ error in  $(y_\nu + \delta_{y_\nu}) - (r_N + \delta_{r_N})$ given $y_\nu + \delta_{y_\nu}$ and $r_N + \delta_{r_N}$.

From (\ref{pm}), (\ref{e2E}) four times, (\ref{y}), (\ref{C3}), and (\ref{rN}):
\begin{align}
\label{C5}
\|\delta_\nu\|_A  
&\le \ewed  (\|y_\nu - r_N\| + \|\delta_{y_\nu}\| + \|\delta_{r_N}\|)  \\
&\le \kappa^\frac{1}{2} \ewed  (\|y\|_A + \|\delta_{y_\nu}\|_A + \|\delta_{r_N}\|_A)  \nonumber \\
&\le C_5 \|A^{-1}r\|_A \nonumber .
\end{align}

\item Line~\ref{tg:postrelax}. $\delta_y = \delta_{y_\nu} - \delta_{r_N} + \delta_\nu$,  accumulated error in  $y \equiv y_\nu - r_N$.

From (\ref{C3}), (\ref{rN}), (\ref{C5}):
\[
\|\delta_y\|_A \le \|\delta_{y_\nu}\|_A + \|\delta_{r_N}\|_A + \|\delta_\nu\|_A \le (C_3 + C_4 + C_5) \|A^{-1}r\|_A .
\]
\end{itemize}
\end{proof}

\begin{rem}
While the bound for $\delta_{\rho_{tg}}$ in (\ref{conv}) is of course very complicated, it can, however, be used to predict the effects of specific precisions under the reasonable assumption that the involved parameters can be estimated. For uniform discretization and refinement of uniformly elliptic PDEs, all of the terms in the $C_k$ other than $\ewed$ and $\kappa^\frac{1}{2}$ are typically $O(1)$. The condition numbers in this case typically grow in proportion to a power of the inverse of the mesh size (and, therefore, $n$ and $n_c$). Thus, since the objective is to choose $\ewed$ so small that $\rho_{tg} \ll 1$, it follows that the principal term in $\delta_{\rho_{tg}}$ is $\pvard \equiv \kappa^\frac{1}{2} \ewed$. Given the {\em progressive precision} objective of determining $\ewed$ for any given $n$ so that $\pvard$ is roughly constant, it is then easy to see that $\delta_{\rho_{tg}} = O(\pvard)$.

To be more specific, under assumptions that are typical of discrete elliptic equations, suppose that $A$ represents an element of a hierarchy of matrices where $\kappa$ grows with $n$ and $\ewed$ is chosen to ensure that $\rho_{tg}$ is suitably smaller than $1$. A close inspection of the $C_k$ reveals that theoretical convergence requires $\pvard$ to be substantially less than $1$. This in turn implies that $\ewed$ is substantially less than $1$ and must decrease as $n$ grows.  With this in mind, we can separate out the the linear term in $\pvard$ from each constant, while the rest would then involve combinations of $\ewed, \ewed^2, \pvard^2,$ and higher powers. Since the linear terms dominate, if we eliminate $C_0$ because of its dependence only on $\ewed$ and let $\xi = \left(\frac{\kappa_c}{\kappa}\right)^\frac{1}{2}$, then we can write $C_k = \gamma_k \pvard +O(\pvard^2), 1 \le k \le 5,$ where the constants are given by
\begin{align*}
\gamma_1 &= \xi (\eta_P (1 + \aladot (1 + \eta_A \eta_M) + \eta_A (\eta_M + \alpha_M)) + \alpdot \eta_P (1 + \eta_A \eta_M)) ,\\
\gamma_2 &= 2 \alpdot \eta_P + 2 \gamma_1, \\
\gamma_3 &= \xi \gamma_2 + 2 + \eta_M, \\
\gamma_4 &= \eta_A \gamma_3 + \aladot (2 \eta_A + 1), \\
\gamma_5 &= 2.
\end{align*}
If we now define the various quantities in the definition of the $\gamma_k$ as bounds relevant to each $A$ and $P$ in the hierarchy and we choose $\pvard$ so that it is roughly constant, then we can say that $\|\delta_{\rho_{tg}}\|_A$ is in fact bounded by a polynomial in $\pvard$ with constant coefficients. The above estimates show that we can choose $\pvard$ to be a fixed value less than $1$ to guarantee that the two-grid cycle converges optimally in the sense that $\rho_{tg}$ is bounded uniformly less than 1.
\label{constants}
\end{rem}

\section{Conclusion}
\label{conclusion}

The theory in \cite{McCormick2021} established optimal error bounds for iterative refinement, $V(1,0)$-cycles, and FMG based on the parameters $\pvard = \kappa^\frac{1}{2} \ewed$, $\pvar = \kappa^\frac{1}{2} \ewe$, and $\pvarb = \kappa \eweb$. The theory in \cite{Tamstorf2021} extends these results to include multiple relaxation sweeps, quantization of $A$ and $P$, and rounding errors in forming $A_c$. The crux of the theory in \cite{McCormick2021} is the bound of the $\ewed$-precision error for the two-grid cycle with just one pre-relaxation step. Theorem~\ref{thm:tg} above extends this two-grid theory to also allow for post-relaxation. The proof in \cite{McCormick2021} for the $V(1,0)$-cycle relies solely on the error being bounded in energy by a polynomial in $\pvard$. Our theory here achieves the same type of bound, which means that we can appeal to \cite{McCormick2021} and \cite{Tamstorf2021} to achieve all of the same extensions established there. Note, for example, that our theory thus confirms optimal error estimates for general progressive precision $V(\mu,\nu)$-cycles with multiple sweeps and FMG that uses them.

\bibliographystyle{siamplain}
\bibliography{../mpmg/mpmg}
\end{document}